\documentclass{article}

\usepackage[utf8]{inputenc}
\usepackage[T1]{fontenc}
\usepackage{hyperref}
\usepackage{url}
\usepackage{booktabs}
\usepackage{amsfonts}
\usepackage{nicefrac}
\usepackage{microtype}
\usepackage{lipsum}
\usepackage{graphicx}
\usepackage{amsmath}
\usepackage{amssymb}
\usepackage{amsthm}
\usepackage{enumitem}
\usepackage[all]{xy}
\SelectTips{cm}{12}

\newtheorem{lemma}{Lemma}
\newtheorem*{lemma*}{Lemma}

\newtheorem{theorem}{Theorem}
\newtheorem*{theorem*}{Theorem}

\theoremstyle{definition}
\newtheorem{example}{Example}

\DeclareMathOperator\Ker{Ker}

\DeclareMathOperator{\id}{id}

\DeclareMathOperator{\Image}{Im}

\newcommand{\up}[2]{{^{#1}\!{#2}}}

\newcommand{\Set}{\mathbf{Set}}
\DeclareMathOperator{\Pro}{Pro}
\newcommand{\freeact}{\mathbin{\flat}}

\title{Actions of pro-groups and pro-rings}

\author{
  Egor Voronetsky\thanks{Research is supported by the Russian Science Foundation grant 19-71-30002.} \\
  Chebyshev Laboratory, \\
  St. Petersburg State University, \\
  14th Line V.O., 29B, \\
  Saint Petersburg 199178 Russia \\
}

\begin{document}
\maketitle

\begin{abstract}
We give an explicit description of internal actions in the semi-abelian categories of pro-groups and non-unital pro-rings in terms of actions of group objects and ring objects in \(\Pro(\Set)\), as well as in some related categories. Also, we show that a similar result fails for Lie algebras.
\end{abstract}

\section{Introduction}

Pro-completions of the categories of algebraic objects are widely used in algebraic geometry \cite{Etale}, algebraic topology \cite{Shape}, and recently in algebraic \(K\)-theory \cite{CentralityK2}. From the main result of \cite{Janelidze} it follows that the categories of pro-groups \(\Pro(\mathbf{Grp})\), non-unital associative pro-rings \(\Pro(\mathbf{Rng})\), and Lie pro-\(K\)-algebras \(\Pro(\mathbf{Lie}_K)\) over a commutative ring \(K\) are semi-abelian, i.e. there are naturally defined internal actions in these categories. For example, we may consider algebras over a given pro-ring \(R\), i.e. pro-rings with an action of \(R\).

In this paper we answer to the natural question, how to describe actions of a given pro-object \(G\) on a pro-object \(X\), in the cases of pro-groups and pro-rings, in theorems \ref{act-pro-group}--\ref{act-pro-crng}. Namely, actions of \(G\) on \(X\) in the sense of semi-abelian categories (i.e. the isomorphism classes of split extensions of \(G\) by \(X\)) are the same as classical actions of \(G\) on \(X\) as algebraic objects in \(\Pro(\Set)\), i.e. they may be defined by finitely many morphisms between pro-sets satisfying some axioms. Our approach fails for Lie pro-algebras, we provide a counterexample.

For example, our main result (theorem \ref{act-pro-group}) implies that any pro-group with an automorphism of a finite order \(n\) is isomorphic to an inverse system of groups with automorphisms of order \(n\).

We mainly refer to \cite{BB} for the semi-abelian categories and to \cite[chapter I, \S 1]{Shape} for the pro-completions.

\section{Pro-sets}

If \(\mathcal C\) is a category, then we write \(X \in \mathcal C\) if \(X\) is an object of \(\mathcal C\). For any algebraic theory \(T\) and a cartesian category \(\mathcal C\) we denote by \(T(\mathcal C)\) the category of \(T\)-objects in \(\mathcal C\) and homomorphisms between them. Also, the category of \(T\)-algebras \(T(\Set)\) is denoted by \(\mathbf T\). We need the algebraic theories
\begin{itemize}
\item \(\mathrm{Grp}\) of groups;
\item \(\mathrm{Ab}\) of abelian groups;
\item \(\mathrm{Rng}\) of non-unital associative rings (we call them rings for simplicity);
\item \(\mathrm{Ring}\) of unital rings;
\item \(\mathrm{CRng}\) of commutative rings;
\item \(\mathrm{CRing}\) of unital commutative rings;
\item \(\mathrm{Mod}_K\) of modules over a unital commutative ring \(K\);
\item \(\mathrm{Lie}_K\) of Lie algebras over a unital commutative ring \(K\).
\end{itemize}
Also, we denote by \(\mathrm{Alg}(\mathcal C)\) the category of unital algebras in \(\mathcal C\). Namely, objects of \(\mathrm{Alg}(\mathcal C)\) are the morphisms \(f \colon K \to R\) in \(\mathrm{Ring}(\mathcal C)\) such that \(f\) is a monomorphism in \(\mathcal C\) and \(f\) is central, i.e. the ring \(\mathcal C(U, K)\) is central in \(\mathcal C(U, R)\) for any \(U \in \mathcal C\). Morphisms from \(f \colon K \to R\) to \(f' \colon K' \to R'\) are the morphisms \(h \in \mathrm{Ring}(\mathcal C)(R, R')\) such that \(h \circ f\) factors through \(f'\). As for algebraic theories, \(\mathbf{Alg} = \mathrm{Alg}(\Set)\).

Recall that a category \(\mathcal I\) is called \textit{filtered} if it is non-empty, for any two objects \(i, j \in \mathcal I\) there is an object \(k \in \mathcal I\) with morphisms \(\varphi \colon i \to k\) and \(\psi \colon j \to k\), and for any parallel morphisms \(\varphi, \psi \colon i \to j\) in \(\mathcal I\) there is a morphism \(\theta \colon j \to k\) such that \(\theta \circ \varphi = \theta \circ \psi\). A functor \(u \colon \mathcal I \to \mathcal J\) between filtered categories is called \textit{cofinal} if for any \(j \in \mathcal J\) there is \(i \in \mathcal I\) and a morphism \(j \to u(i)\). A filtered category \(\mathcal I\) is called a \textit{directed set} if it is small, skeletal, and for any two objects \(i, j \in \mathcal I\) there is at most one morphism \(i \to j\), so we may consider \(\mathcal I\) as a poset. A directed set \(I\) is called \textit{cofinite} if for any \(i \in I\) there is only finitely many elements less than \(i\). If \(\mathcal I\) is a small filtered category, then there exists a cofinite directed set \(J\) with a cofinal functor \(J \to \mathcal I\), see \cite[chapter I, \S 1.4, theorem 4]{Shape}.

An \textit{inverse system} \(X\) in a category \(\mathcal C\) is a contravariant functor \(\mathcal I_X \to \mathcal C\), where \(\mathcal I_X\) is a small filtered category. A \textit{morphism} of inverse systems \(f \colon X \to Y\) consists of a map \(f^*\) from the set of objects of \(\mathcal I_Y\) to the set of objects of \(\mathcal I_X\) and morphisms \(f_i \colon X_{f^*(i)} \to Y_i\) for all \(i \in \mathcal I_Y\) such that for all \(\varphi \in \mathcal I_Y(i, j)\) there is sufficiently large \(k \in \mathcal I_X\) with morphisms \(\psi \in \mathcal I_X(k, f^*(i))\) and \(\theta \in \mathcal I_X(k, f^*(j))\) satisfying \(f_i \circ X_\psi = Y_\varphi \circ f_j \circ X_\theta \colon X_k \to Y_i\). Clearly, the inverse systems and their morphisms form a category. A \textit{pro-completion} \(\Pro(\mathcal C)\) is its factor-category by the following equivalence relation: two morphisms \(f, g \colon X \to Y\) of inverse systems are \textit{equivalent} if for all \(i \in \mathcal I_Y\) there is sufficiently large \(j \in \mathcal I_X\) and morphisms \(\varphi \in \mathcal I_X(j, f^*(i))\), \(\psi \in \mathcal I_X(j, g^*(i))\) such that \(f_i \circ X_\varphi = g_i \circ X_\psi \colon X_j \to Y_i\). By \cite[chapter I, \S 1.1, remark 4]{Shape},
\[\Pro(C)(X, Y) = \varprojlim_{j \in \mathcal I_Y} \varinjlim_{i \in \mathcal I_X} \mathcal C(X_i, Y_j).\]

If \(X\) is an inverse system in \(\mathcal C\) and \(u \colon \mathcal J \to \mathcal I_X\) is a cofinal functor, then the induced morphism of inverse systems \(f \colon X \to u^* X = (X \circ u)\) given by \(f^*(j) = u(j)\) and \(f_j = \id_{X_{u(j)}}\) is an isomorphism in \(\Pro(\mathcal C)\). In particular, any object in \(\Pro(\mathcal C)\) is isomorphic to an inverse system indexed by a cofinite directed set. If \(\mathcal I_X\) is a directed set, \(i \leq j\) are two indices in \(\mathcal I_X\), and \(x \in X_j\), then we denote the image of \(x\) in \(X_i\) by \(x|_i\).

A morphism \(f \colon X \to Y\) of inverse systems is called a \textit{level morphism} if \(\mathcal I_X = \mathcal I_Y\), \(f^*(i) = i\) for all \(i\), and \(Y_\varphi \circ f_i = f_j \circ X_\varphi \colon X_i \to Y_j\) for all \(\varphi \in \mathcal I_X(i, j)\). Each morphism in \(\Pro(\mathcal C)\) is isomorphic to a level morphism, i.e. for any \(f \colon X \to Y\) there is a level morphism \(f' \colon X' \to Y'\) and isomorphisms \(X \to X'\), \(Y \to Y'\) in \(\Pro(\mathcal C)\) making the resulting square commutative in \(\Pro(\mathcal C)\), see \cite[chapter I, \S 1.3, theorem 3]{Shape}. Moreover, here we may assume that \(X \to X'\) and \(Y \to Y'\) are induced by cofinal morphisms from a cofinite directed poset. By \cite[chapter II, \S 2.2, theorem 5]{Shape}, a level morphism \(f \colon X \to Y\) of inverse systems is an isomorphism if and only if for any \(i \in \mathcal I_X\) there are \(j \in \mathcal I_X\), \(\varphi \colon i \to j\), and \(u \colon Y_j \to X_i\) such that \(f_i \circ u = Y_\varphi\) and \(u \circ f_j = X_\varphi\).

Any object \(X \in \mathcal C\) may be considered as the inverse system over the terminal category with only one index \(*\). The canonical embedding functor \(\mathcal C \to \Pro(\mathcal C)\) is fully faithful and every inverse system \(X\) is the projective limit of itself in \(\Pro(\mathcal C)\). We denote the projective limit of an inverse system \(X\) in \(\Pro(\mathcal C)\) by \(\varprojlim_{i \in \mathcal I_X}^{\Pro(\mathcal C)} X_i\), where \(X_i\) are themselves some pro-objects, all such limits necessary exist.

By \cite[corollary 6.1.17, proposition 6.1.18]{KSch}, if \(\mathcal C\) is finitely complete or finitely cocomplete, then \(\Pro(\mathcal C)\) is also finitely complete or finitely cocomplete respectively and the limits and colimits of finite level diagrams may be computed levelwise. Also, if \(f \colon X \to Y\) is a level morphism of inverse systems such that all \(f_i\) are monomorphisms or epimorphisms in \(\mathcal C\), then \(f\) is a monomorphism or an epimorphism in \(\Pro(\mathcal C)\) by \cite[propositions 2.3 and 2.4]{Dydak}.

Now let us recall the definition of regular categories from \cite{Handbook}. A \textit{categorical equivalence relation} on an object \(X\) in a finitely complete category \(\mathcal C\) is an object \(R\) with a monomorphism \(R \to X \times X\) such that \(\mathcal C(U, R)\) is an equvalence relation on \(\mathcal C(U, X)\) for all \(U \in \mathcal C\). If \(f \in \mathcal C(X, Y)\), then its \textit{kernel pair} \(R = \lim(X \xrightarrow{f} Y \xleftarrow{f} X)\) is a categorical equivalence relation on \(X\). A morphism \(f\) is called a \textit{regular epimorphism} if it is the coequalizer of a kernel pair considered as a pair of parallel morphisms to the source of \(f\). A finitely complete category \(\mathcal C\) is called \textit{regular} if all kernel pairs of morphisms have coequalizers and regular epimorphisms are preserved under pullbacks. In a regular category every morphism \(f \colon X \to Y\) admits an \textit{image decomposition} \(X \to \Image(f) \to Y\), where \(X \to \Image(f)\) is a regular epimorphism and \(\Image(f) \to Y\) is a monomorphism, such a decomposition is unique up to a unique isomorphism and functorial on \(f\). A regular category \(\mathcal C\) is called \textit{Barr exact} if, in addition, all categorical equivalence relations in \(\mathcal C\) are kernel pairs.

\begin{lemma} \label{image-dec}
Let \(\mathcal C\) be a regular category and \(f \colon X \to Y\) be a level morphism of inverse systems in \(\mathcal C\) with the index category \(\mathcal I\). Then \(X \to \varprojlim_{i \in \mathcal I} \Image(f_i) \to Y\) is the image decomposition of \(f\) in \(\Pro(\mathcal C)\). In particular, if all \(f_i\) are regular epimorphisms in \(\mathcal C\), then \(f\) is a regular epimorphism in the pro-completion.
\end{lemma}
\begin{proof}
Since the image decomposition is functorial, the pro-object \(Z = \varprojlim_{i \in \mathcal I} \Image(f_i)\) is well-defined. Now the morphism \(Z \to Y\) is a monomorphism as the formal projective limit of monomorphisms \(\Image(f_i) \to Y_i\) in \(\mathcal C\). The kernel pair of \(f\) in \(\Pro(\mathcal C)\) is \(R = \varprojlim_{i \in \mathcal I} R_i\), where \(R_i\) is the kernel pair of \(f_i\) in \(\mathcal C\). It remains to check that \(X \to Z\) is the coequalizer of the morphisms \(R \rightrightarrows X\). But this easily follows from the formula for the set of morphisms in \(\Pro(\mathcal C)\) since projective and direct limits of sets commute with equalizers.
\end{proof}

Recall that the monomorphisms in \(\Set\) and in all the categories of algebraic objects \(\mathbf T\) listed above (excluding \(\mathbf{Alg}\)) are precisely the injective maps, the regular epimorphisms are the surjective maps. By \cite[example 1.11]{Janelidze}, the pro-completion of a regular category is also regular. It follows that the category of pro-sets \(\Pro(\Set)\) is regular and finitely cocomplete, and similarly for various \(\Pro(\mathbf T)\). Unlike \(\Set\), the category \(\Pro(\Set)\) is not Barr exact, as the following example shows.

\begin{example}
Let \(R_n = \{(x, y) \in \mathbb R^2 \mid x - 1/n < y < x + 1/n\}\) for \(n > 0\) and \(R = \varprojlim_n^{\Pro(\Set)} R_n\) be their formal projective limit. Then \(R\) is a categorical equivalence relation on \(\mathbb R\) in \(\Pro(\Set)\) since all \(R_n\) are reflexive symmetric relations on \(\mathbb R\) and \(R_{2n} \circ R_{2n} \subseteq R_n\). Now suppose that \(f \colon \mathbb R \to X\) is a morphism of inverse systems with the kernel pair containing \(R\) in the category \(\Pro(\Set)\). This means that for any \(i \in \mathcal I_X\) there is \(n > 0\) such that \(f_i(x) = f_i(y)\) for all \((x, y) \in R_n\). Then it easily follows that all \(f_i\) have one-element images, i.e. the kernel pair of \(f\) is \(\mathbb R^2\). On the other hand, \(R \to \mathbb R^2\) is not an isomorphism of pro-sets since the embeddings \(R_n \to \mathbb R^2\) of sets do not have sections. In other words, \(R\) is a categorical equivalence relation but not a kernel pair in \(\Pro(\Set)\).
\end{example}

We also need the notion of homological and semi-abelian categories from \cite[sections 4 and 5]{BB}. A category \(\mathcal C\) is called \textit{pointed} if there is a zero object \(0\), i.e. it is simultaneously initial and terminal. A \textit{split extension} of \(Z\) by \(X\) in a finitely complete pointed category \(\mathcal C\) is a diagram of type
\[\xymatrix@R=24pt@C=48pt@!0{
X \ar@{>->}[r]^{i} & Y \ar@{->>}@<-3pt>[r]_{p} & Z \ar@{>->}@<-3pt>[l]_{s},
}\]
where \(p \circ s = \id_Z\) and \(i\) is the \textit{kernel} of \(p\), i.e. the equalizer of \(p\) and \(0 \colon X \to Y\). The object \(Y\) is called a \textit{semi-direct product} of \(X\) and \(Z\), it is also denoted by \(X \rtimes Z\). If
\[\xymatrix@R=24pt@C=48pt@!0{
X \ar@{>->}[r]^{i'} & Y' \ar@{->>}@<-3pt>[r]_{p'} & Z \ar@{>->}@<-3pt>[l]_{s'}
}\]
is another split extension and \(f \in \mathcal C(Y, Y')\) satisfies \(f \circ i = i'\), \(f \circ s = s'\), and \(p' \circ f = p\), then \(f\) is called a \textit{morphism} of split extensions. A finitely complete pointed category \(\mathcal C\) is called \textit{protomodular} if every morphism between split extensions of any objects is an isomorphism. A category \(\mathcal C\) is called \textit{homological} if it is pointed, regular, and protomodular (in particular, it is finitely complete). If, in addition, \(\mathcal C\) is Barr exact and finitely cocomplete, then it is called \textit{semi-abelian}.

If \(X\) and \(G\) are objects of a homological category \(\mathcal C\), then an \textit{action} of \(G\) on \(X\) is an isomorphism class of split extensions \(X \to Y \leftrightarrows G\). Each object \(G\) canonically acts on itself via
\[\xymatrix@R=24pt@C=48pt@!0{
G \ar@{>->}[r]^(0.4){i_1} & G \times G \ar@{->>}@<-3pt>[r]_(0.6){p_2} & G \ar@{>->}@<-3pt>[l]_(0.4){\Delta},
}\]
where \(i_1 \colon G \cong G \times 0 \to G \times G\) is the left embedding, \(p_2 \colon G \times G \to 0 \times G \cong G\) is the right projection, and \(\Delta \colon G \to G \times G\) is the diagonal embedding. If \(G\) and \(X\) are any objects in a semi-abelian category, then \(G\) canonically acts on \(G \freeact X = \Ker(G \amalg X \to G \amalg 0 \cong G)\) and the actions of \(G\) on \(X\) are in a one-to-one correspondence with the morphisms \(G \freeact X \to X\) satisfying some additional conditions, see \cite{Actions} for details. In particular, there is only a set of possible actions of \(G\) on \(X\).

The categories \(\mathbf{Grp}\), \(\mathbf{Ab}\), \(\mathbf{Rng}\), \(\mathbf{CRng}\), \(\mathbf{Lie}_K\), and \(\mathbf{Mod}_K\) are semi-abelian. By \cite[the list after example 1.12]{Janelidze}, the pro-completions of these categories (or any other semi-abelian category) are also semi-abelian.

\section{Pro-groups and pro-rings}

Let \(T \neq \mathrm{Alg}\) be one of the algebraic theories listed above. The category \(\mathbf T\) may be extended to the case of pro-sets in two canonical ways, namely, to its pro-completion \(\Pro(\mathbf T)\) and to the category \(T(\Pro(\Set))\). Both these categories are regular, in \(T(\Pro(\Set))\) monomorphisms, regular epimorphisms, and the image decompositions are the same as in \(\Pro(\Set)\). Clearly, there are canonical functors \(\mathbf T \to \Pro(\mathbf T) \to T(\Pro(\Set))\), the left one is fully faithful.

\begin{lemma} \label{pro-group}
In the diagram
\[\xymatrix@R=24pt@C=72pt@!0{
\Pro(\mathbf{Grp}) \ar[r] & \mathrm{Grp}(\Pro(\Set)) \\
\Pro(\mathbf{Ab}) \ar[u] \ar[r] & \mathrm{Ab}(\Pro(\Set)) \ar[u]
}\]
all functors are fully faithful. A pro-group \(G\) is isomorphic to an abelian pro-group if and only if it is an abelian group object in \(\Pro(\Set)\).
\end{lemma}
\begin{proof}
Clearly, the horizontal functors are faithful and the vertical ones are fully faithful, so in order to prove the first claim it suffice to check that the upper functor is full. Let \(G\) and \(H\) be pro-groups and \(f \colon G \to H\) be a morphism of inverse systems of sets such that it is a homomorphism of group objects in \(\Pro(\Set)\). This means that for every \(i \in \mathcal I_H\) there is \(j \in \mathcal I_G\) and \(\varphi \colon f^*(i) \to j\) such that \(f_i \circ G_\varphi \colon G_j \to H_i\) is a group homomorphism. In other words, \(f\) is equivalent to a morphism of inverse systems of groups.

Now let \(G\) be a pro-group such that it is an abelian group object in \(\Pro(\Set)\). Let \(G'_i = G_i / [G_i, G_i]\) be its levelwise abelianization and \(f \colon G \to G'\) be the corresponding level morphism. Since \(G\) is an abelian group object, for any \(i \in \mathcal I_G\) there are \(j \in \mathcal I_G\) and \(\varphi \colon i \to j\) such that \(G_\varphi(xy) = G_\varphi(yx)\), i.e. \(G_\varphi = u_i \circ f_j\) for some \(u_i \colon G'_j \to G_i\). Now it is easy to see that \(f\) is an isomorphism of pro-groups with the inverse given by the family of \(u_i\).
\end{proof}

In turns out that the canonical functor \(\Pro(\mathbf{Grp}) \to \mathrm{Grp}(\Pro(\Set))\) is not an equivalence.

\begin{example}
Let \(X = \varprojlim^{\Pro(\Set)}_{n > 0} X_n\), where \(X_n = (-1/n, 1/n)\) are the intervals in \(\mathbb R\) and the maps between them are the inclusions. It has an abelian group object structure:
\begin{itemize}
\item The neutral element is \(z \colon \{*\} \to X\), \(z^*(n) = *\), \(z_n(*) = 0 \in X_n\).
\item The inverse is \(i \colon X \to X\), \(i^*(n) = n\), \(i_n(x) = -x\).
\item The addition is \(s \colon X \times X \to X\), \(s^*(n) = 2n\), \(s_n(x, y) = x + y\).
\end{itemize}
Suppose that \(f \colon G \to X\) is a homomorphism of group objects in \(\Pro(\Set)\) such that \(G\) is a pro-group. This means, in particular, that for every \(n\) there is \(i \in \mathcal I_G\) and \(\varphi \colon f^*(n) \to i\) such that \(f_n(G_\varphi(gg')) = f_n(G_\varphi(g)) + f_n(G_\varphi(g'))\). But since \(-1/n < f_n(G_\varphi(g^k)) < 1/n\) for all \(k \geq 1\), we necessarily have \(f_n(G_\varphi(g)) = 0\) for all \(g \in G_i\). In other words, \(f\) is equivalent to the trivial homomorphism. On the other hand, \(X\) is not trivial, so it is not isomorphic to a pro-group.
\end{example}

\begin{lemma} \label{pro-lie}
Let \(K\) be a localization of a finitely generated unital commutative ring such as \(\mathbb Z\), \(\mathbb Q\), or a finite field. Then in the diagram
\[\xymatrix@R=24pt@C=84pt@!0{
\Pro(\mathbf{Lie}_K) \ar[r] & \mathrm{Lie}_K(\Pro(\Set)) \\
\Pro(\mathbf{Mod}_K) \ar[u] \ar[r] & \mathrm{Mod}_K(\Pro(\Set)) \ar[u]
}\]
all functors are fully faithful, where the vertical functors add the zero Lie bracket. A Lie pro-\(K\)-algebra \(G\) is isomorphic to a pro-\(K\)-module if and only if it is an abelian Lie \(K\)-algebra object in \(\Pro(\Set)\).
\end{lemma}
\begin{proof}
The proof of the first claim is the same as in lemma \ref{pro-group} since a pro-set morphism between Lie pro-\(K\)-algebras or pro-\(K\)-modules is a homomorphism if and only if it preserves finitely many operations, namely, the abelian group operations, the Lie bracket in the Lie algeba case, and the multiplications by the finitely generators of a ring \(\widetilde K\) such that \(K\) is a localization of \(\widetilde K\).

Suppose that \(L\) is a Lie pro-\(K\)-algebra such that it is an abelian Lie \(K\)-algebra object in \(\Pro(\Set)\). Let \(L'_i = L_i / [L_i, L_i]\) be its levelwise abelianization and \(f \colon L \to L'\) be the corresponding level morphism. Since \(L\) is an abelian Lie \(K\)-algebra object, for any \(i \in \mathcal I_L\) there are \(j \in \mathcal I_L\) and \(\varphi \colon i \to j\) such that \(L_\varphi([x, y]) = 0\), i.e. \(L_\varphi = u_i \circ f_j\) for some \(u_i \colon L'_j \to L_i\). Then \(f\) is an isomorphism of pro-groups with the inverse given by the family of \(u_i\).
\end{proof}

The condition on \(K\) in lemma \ref{pro-lie} is necessary.

\begin{example} \label{k-cond}
Let \(K = \mathbb Z[x_1, x_2, \ldots]\) be a polynomial ring over infinitely many variables. Consider the \(K\)-module \(N = \mathbb Z \oplus \mathbb Z \oplus \ldots\) with the identity actions of \(x_i\) and the \(K\)-modules \(M_k\) with the trivial actions of \(x_i\) for \(k \geq 0\), where \(M_k\) are the subgroups of \(N\) with trivial first \(k\) components. Then \(M = \varprojlim^{\Pro(\mathbf{Ab})}_k M_k\) is a pro-module over \(K\), where the structure maps are the inclusions. Also, \(f \colon M \to N\) is a morphism of abelian pro-groups given by the identity homomorphism \(M_0 \to N\). Clearly, \(f\) is not equivalent to a morphism of pro-\(K\)-modules since no \(M_k \to N\) is a homomorphism of \(K\)-modules. On the other hand, \(f\) is a homomorphism of \(K\)-modules in \(\Pro(\Set)\) since for any \(n\) the map \(M_k \to N\) preserves the action of \(x_n\) for \(k \geq n\).
\end{example}

\begin{lemma} \label{pro-ring}
In the diagrams
\[\xymatrix@R=24pt@C=84pt@!0{
\Pro(\mathbf{Rng}) \ar[r] & \mathrm{Rng}(\Pro(\Set)) &
\Pro(\mathbf{Ring}) \ar[r] & \mathrm{Ring}(\Pro(\Set)) \\
\Pro(\mathbf{CRng}) \ar[u] \ar[r] & \mathrm{CRng}(\Pro(\Set)) \ar[u] &
\Pro(\mathbf{CRing}) \ar[r] \ar[u] & \mathrm{CRing}(\Pro(\Set)) \ar[u] \\
}\]
all functors are fully faithful. A pro-ring \(R\) is isomorphic to a commutative pro-ring, unital pro-ring, or unital commutative pro-ring if and only if it is a commutative ring object, unital ring object, or unital commutative ring object in \(\Pro(\Set)\) respectively.
\end{lemma}
\begin{proof}
The first claim may be proved as in lemma \ref{pro-group}. Take a pro-ring \(R\). If it is a commutative ring object in \(\Pro(\Set)\), then it is isomorphic to the commutative pro-ring \(R' = \varprojlim^{\Pro(\mathbf{CRng})}_{i \in \mathcal I_R} R_i / [R_i, R_i]\), where \([R_i, R_i]\) is the ideal of \(R_i\) generated by the additive commutators. If \(R\) is a unital ring object in \(\Pro(\Set)\), then it is isomorphic to the unital pro-ring \(R'' = \varprojlim^{\Pro(\mathbf{Ring})}_{i \in \mathcal I_R} R_i / J_i\), where \(J_i\) is the ideal of \(R_i\) generated by the elements \(ex - x\) and \(xe - x\) for \(x \in R_i\) and \(e \in R_i\) is the element given by the identity \(\{*\} \to R\) in \(\Pro(\Set)\). Finally, if \(R\) is commutative pro-ring and unital commutative ring object in \(\Pro(\Set)\), then the last construction gives a unital commutative pro-ring \(R''\).
\end{proof}

\begin{lemma} \label{pro-alg}
The functor \(\Pro(\mathbf{Alg}) \to \mathrm{Alg}(\Pro(\Set))\) is fully faithful. An object \(f \colon K \to R\) in \(\mathrm{Alg}(\Pro(\Set))\) is isomorphic to a unital pro-algebra if and only if \(K\) and \(R\) are isomorphic to pro-rings.
\end{lemma}
\begin{proof}
Clearly, the given functor is faithful. By lemma \ref{pro-ring}, a morphism in \(\mathrm{Alg}(\Pro(\Set))\) between unital pro-algebras may be represented by a diagram
\[\xymatrix@R=24pt@C=48pt@!0{
K \ar[d] \ar[r]^f & R \ar[d] \\
K' \ar[r]^{f'} & R',
}\]
of inverse systems of unital rings, where both rows are level morphisms and central at each level, and the square is commutative in \(\Pro(\mathbf{Ring})\). Changing the vertical morphisms by equivalent ones, we may assume that the square is commutative in the category of inverse systems, so it represents a morphism in \(\Pro(\mathbf{Alg})\). This proves the first claim.

Now let \(K \to R\) be an object of \(\mathrm{Alg}(\Pro(\Set))\) represented by a morphism \(f \colon K \to R\) of inverse systems of sets. Suppose that \(K\) and \(R\) are pro-rings. Then by lemma \ref{pro-ring} we may assume that \(f\) is a level morphism of inverse systems of unital pro-rings. Now let \(R'_i = R_i / J_i\), where \(J_i\) is the ideal generated by the additive commutators \([x, f_i(y)]\) for \(x \in R_i\) and \(y \in K_i\), and \(K'_i\) be the image of \(f_i(K_i)\) in \(R'_i\). Clearly, the induced morphism \(K' \to R'\) of inverse systems is an object of \(\Pro(\mathbf{Alg})\) isomorphic to \(f\).
\end{proof}

Let \(T\) be an algebraic theory from one of the lemmas \ref{pro-group}--\ref{pro-ring}. Then a morphism \(f \colon X \to Y\) in \(\Pro(\mathbf T)\) is an isomorphism if and only if it is an isomorphism in \(\Pro(\Set)\). Since \(\Pro(\mathbf T) \to \Pro(\Set)\) preserves the image decompositions of morphisms, a morphism \(f \colon X \to Y\) is a monomorphism or a regular epimorphism in \(\Pro(\mathbf T)\) if and only if it is a monomorphism or a regular epimorphism in \(\Pro(\Set)\) respectively. Also, \(\Pro(\mathbf T) \to \Pro(\Set)\) preserves finite limits since the same holds for \(\mathbf T \to \Set\).

\section{Actions of pro-groups}

Let us recall the classical definitions of the actions:
\begin{itemize}
\item A group \(G\) acts on a group \(X\) if there is a map \(a \colon G \times X \to X\) such that \(a(gh, x) = a(g, a(h, x))\), \(a(1, x) = x\), and \(a(g, xy) = a(g, x)\, a(g, y)\).
\item A Lie \(K\)-algebra \(L\) acts on a Lie \(K\)-algebra \(M\) if there is a \(K\)-bilinear map \(a \colon L \times M \to M\) such that \(a(x, [u, v]) = [a(x, u), v] + [u, a(x, v)]\) and \(a([x, y], u) = a(x, a(y, u)) - a(y, a(x, u))\).
\item Any abelian group \(A\) acts on any abelian group \(B\) in the unique (trivial) way inside \(\mathbf{Ab}\) since this category is abelian, the same holds for \(\mathbf{Lie}_K\).
\item A ring \(R\) acts on a ring \(S\) if \(S\) is a non-unital \(R\)-\(R\)-bimodule in such a way that \(p(ab) = (pa)b\), \((ap)b = a(pb)\), and \(a(bp) = (ab)p\) for \(a, b \in S\) and \(p \in R\). In other words, the action is given by the multiplication maps \(l \colon R \times S \to S\) and \(r \colon S \times R \to S\) satisfying some axioms.
\item A commutative ring \(K\) acts on a commutative ring \(S\) if \(S\) is a non-unital \(K\)-module such that \(k(ab) = (ka)b\) for \(a, b \in S\) and \(k \in K\).
\item A unital ring \(R\) acts on a commutative ring \(S\) if \(R\) acts on \(S\) as a ring and \(1a = a = a1\) for all \(a \in S\).
\item A unital commutative ring \(K\) acts on a commutative ring \(S\) if \(K\) acts on \(S\) as a commutative ring and \(1a = a\) for \(a \in S\).
\item A unital algebra \(K \subseteq R\) acts on a ring \(S\) if \(R\) acts on \(S\) as a unital ring and \(ka = ak\) for \(a \in S\) and \(k \in K\).
\end{itemize}
In the first four cases the semi-direct product in the sense of semi-abelian categories is the classical semi-direct product. The coproducts in some of these categories and the objects \(G \freeact X\) may be constructed in the following way:
\begin{itemize}
\item In the case of groups
\[G \amalg H = \{1\} \sqcup G^\# \sqcup H^\# \sqcup (G^\# \times H^\#) \sqcup (H^\# \times G^\#) \sqcup (G^\# \times H^\# \times G^\#) \sqcup (H^\# \times G^\# \times H^\#) \sqcup \ldots\]
as a set, where \(G^\# = G \setminus \{1\}\) and \(H^\# = H \setminus \{1\}\), and \(G \freeact H\) is the abstract group with the generators \(\up gh\) for \(g \in G\) and \(h \in H\) and the relations \(\up g{(hh')} = \up gh\, \up g{h'}\).
\item In the case of abelian groups or \(K\)-modules \(A \amalg B = A \times B\) and \(A \freeact B = B\).
\item In the case of rings
\[R \amalg S = R \oplus S \oplus (R \otimes S) \oplus (S \otimes R) \oplus (R \otimes S \otimes R) \oplus (S \otimes R \otimes S) \oplus \ldots,\]
where \(\otimes\) means the tensor product over \(\mathbb Z\), and
\[R \freeact S = S \oplus (R \otimes S) \oplus (S \otimes R) \oplus (R \otimes S \otimes R) \oplus (S \otimes R \otimes S) \oplus \ldots.\]
\item In the case of commutative rings \(R \amalg S = R \oplus S \oplus (R \otimes S)\) and \(R \freeact S = S \oplus (R \otimes S)\).
\end{itemize}

Let \(T\) be the algebraic theory of groups, Lie \(K\)-algebras, rings, or commutative rings. Using Yoneda embedding of \(\Pro(\Set)\) to the category of presheaves on \(\Pro(\Set)\), it is easy to see that actions in \(T(\Pro(\Set))\) may be described via the same operations and axioms as in \(\mathbf T\). Moreover, the canonical morphism of pro-sets \(G \times X \to G \rtimes X\) is always an isomorphism for any action of \(G\) on \(X\). It follows that \(T(\Pro(\Set))\) is homological and there is only a set of actions of \(G\) on \(X\) for any \(G, X \in T(\Pro(\Set))\).

\begin{theorem} \label{act-pro-group}
Let \(G\) and \(X\) be pro-groups and \(a \colon G \times X \to X\) be an action of \(G\) on \(X\) in \(\mathrm{Grp}(\Pro(\Set))\). Then the split extension \(X \to X \rtimes G \leftrightarrows G\) is isomorphic to a formal projective limit \(X' \to X' \rtimes G' \leftrightarrows G'\) of split extensions of groups in such a way that \(G \to G'\) is induces by a cofinal functor between the index categories. The functor \(\Pro(\mathbf{Grp}) \to \mathrm{Grp}(\Pro(\Set))\) induces a bijection between the sets of actions of \(G\) on \(X\) in these categories.
\end{theorem}
\begin{proof}
For simplicity, we denote by \(\gamma(i) \in \mathcal I_G\) and \(\chi(i) \in \mathcal I_X\) the components of \(a^*(i)\) for the action \(a\) and the index \(i \in \mathcal I_X\). Without loss of generality, we may assume that the index categories of \(G\) and \(X\) are cofinite posets. Replacing \(\mathcal I_X\) by \(\mathcal I_G \times \mathcal I_X\) and increasing \(\gamma\) and \(\chi\), we may further assume that
\begin{itemize}
\item \(\gamma\) is monotone and cofinal;
\item \(\chi\) is monotone and \(\chi(i) \geq i\) for all \(i\);
\item \(a_j(g, x)|_i = a_i(g|_{\gamma(i)}, x|_{\chi(i)})\) for \(i \leq j\);
\item \(a_i(g, xy) = a_i(g, x)\, a_i(g, y)\);
\item \(a_i(1, x) = x|_i\);
\item \(a_i(g, a_{\chi(i)}(h, x)) = a_i(g h|_{\gamma(i)}, x|_{\chi(i)})\).
\end{itemize}

Let \(\widetilde X_i = G_{\gamma(i)} \freeact X_{\chi(i)}\) for all \(i \in \mathcal I_X\) and \(\widetilde X\) be the corresponding inverse system. By the above properties, there is a level morphism \(f \colon \widetilde X \to X\) of pro-groups such that \(f_i(\up g x) = a_i(g, x)\) for \(g \in G_{\gamma(i)}\) and \(x \in X_{\chi(i)}\). Now let \(X'_i\) be the factor-group of \(\widetilde X_i\) by the normal \(G_{\gamma(i)}\)-invariant subgroup generated by the elements \(\up{g|_{\gamma(i)}}{(x|_{\chi(i)})}\, a_{\chi(i)}(h, x)^{-1}\), this subgroup lies in the kernel of \(f_i\).

It follows that \(f\) induces a level morphism \(X' \to X\) of pro-groups, it is an isomorphism with the inverse given by the canonical maps \(X_{\chi(i)} \to X'_i\). Let \(G'_i = G_{\gamma(i)}\), it acts on \(X'_i\) by the construction. It is easy to check that the resulting action of \(G'\) on \(X'\) is isomorphic to \(a\) via the isomorphisms \(G \to G'\) and \(X \to X'\).

The second claim follows from lemma \ref{pro-group} and the first claim.
\end{proof}

Unfortunately, the functor \(\Pro(\mathbf{Lie}_K) \to \mathrm{Lie}_K(\Pro(\Set))\) does not induce a bijection between the sets of actions unless \(K = 0\).

\begin{example}
Let \(K\) be any non-zero unital commutative ring. Consider the inverse system \(M\) of \(K\)-modules \(M_n = K^n\) with the zero Lie brackets and the structure homomorphisms
\[M_{n + 1} \to M_n,\, (x_1, \ldots, x_{n + 1}) \mapsto (x_1, \ldots, x_n).\]
The \(K\)-module \(K\) (considered as the abelian Lie \(K\)-algebra) acts on \(M\) in \(\mathrm{Lie}_K(\Pro(\Set))\) by the morphism \(a \colon K \times M \to M\) with
\[a^*(n) = (*, n + 1),\, a_n(k; x_1, \ldots, x_{n + 1}) = (kx_2, \ldots, kx_{n + 1}).\]
So the semi-direct product \(M \rtimes K\) in \(\mathrm{Lie}_K(\Pro(\Set))\) is the inverse limit of \(M_n \times K = K^{n + 1}\) of \(K\)-modules with the Lie bracket
\[K^{n + 1} \times K^{n + 1} \to K^n,\, (x_1, \ldots, x_{n + 1}; k; y_1, \ldots, y_{n + 1}; l) \mapsto (ky_2 - l x_2, \ldots, ky_{n + 1} - l x_{n + 1}; 0).\]
Now suppose that \(f \colon M \rtimes K \to L\) is a homomorphism of Lie \(K\)-algebra objects in \(\Pro(\Set)\), where \(L\) is a Lie pro-\(K\)-algebra. We claim that the induced homomorphism \(M \to L\) is trivial. Without loss of generality, \(L\) is an abstract Lie \(K\)-algebra and \(f\) is given by a \(K\)-linear homomorphism \(f_* \colon M_n \times K \to L\) such that
\[f_*(ky_2 - l x_2, \ldots, ky_{n + 1} - l x_{n + 1}; 0) = [f_*(x_1, \ldots, x_n; k), f_*(y_1, \ldots, y_n; l)].\]
By the \(K\)-linearity it follows that
\[f_*(0, \ldots, y_i, \ldots, 0; 0) = [f_*(0, \ldots, 0; 1), f_*(0, \ldots, y_i, \ldots, 0; 0)]\]
for \(2 \leq i \leq n\) (here \(y_i\) is on the \((i - 1)\)-th position in the left hand side and on the \(i\)-th position in the right hand side) and
\[f_*(0, \ldots, y_{n + 1}; 0) = 0.\]
By induction we get \(f_*(0, \ldots, y_i, \ldots, 0; 0) = 0\) for all \(2 \leq i \leq n + 1\), where \(y_i\) is on the \((i - 1)\)-th position, as claimed. On the other hand, \(M\) is not isomorphic to the zero Lie \(K\)-algebra object in \(\Pro(\Set)\), so \(M \rtimes K\) is not isomorphic to a Lie pro-\(K\)-algebra. 
\end{example}

\section{Actions of pro-rings}

We need to generalize actions of unital rings, commutative rings, and unital algebras to the pro-objects.
\begin{itemize}
\item A \textit{split extension of a unital pro-ring} \(R\) \textit{by a pro-ring} \(S\) is a split extension \(S \to S \rtimes R \leftrightarrows R\) in \(\Pro(\mathbf{Rng})\) such that \(S \rtimes R\) is a unital pro-ring and both right morphisms lie in \(\Pro(\mathbf{Ring})\). A \textit{morphism} of such extensions is a morphism of extensions of pro-rings lying in \(\Pro(\mathbf{Ring})\).
\item A \textit{split extension of a unital commutative pro-ring} \(K\) \textit{by a commutative pro-ring} \(S\) is a split extension \(S \to S \rtimes K \leftrightarrows K\) in \(\Pro(\mathbf{CRng})\) such that \(S \rtimes K\) is a unital commutative pro-ring and both right morphisms lie in \(\Pro(\mathbf{CRing})\). A \textit{morphism} of such extensions is a morphism of extensions of pro-rings lying in \(\Pro(\mathbf{CRing})\).
\item A \textit{split extension of a unital pro-algebra} \(K \rightarrowtail R\) \textit{by a pro-ring} \(S\) is a diagram
\[\xymatrix@R=24pt@C=48pt@!0{
S \ar@{>->}[r]^{i} & S \rtimes R \ar@{->>}@<-3pt>[r]_(0.6){p} & R \ar@{>->}@<-3pt>[l]_(0.4){s} \\
& K' \ar@{>->}[u] \ar[r]^{\sim} & K \ar@{>->}[u],
}\]
where the columns are unital pro-algebras, \(s\) and \(p\) are morphisms of unital pro-algebras inducing an isomorphism \(K \cong K'\), and the top row is a split extension in \(\Pro(\mathbf{Rng})\). A \textit{morphism} of such extensions is a morphism of extensions of pro-rings lying in \(\Pro(\mathbf{Ring})\).
\end{itemize}
If instead of \(\Pro(\mathbf T)\) we use \(T(\Pro(\Set))\) or just \(\mathbf T\), then the resulting isomorphism classes of split extensions are precisely the actions from the previous section defined via operations and axioms.

\begin{theorem} \label{act-pro-alg}
Let \(K \rightarrowtail R\) be a unital pro-algebra, \(S\) be a pro-ring, and \(l \colon R \times S \to S\), \(r \colon S \times R \to S\) be an action of the unital algebra object \(K \rightarrowtail R\) on the ring object \(S\) in \(\Pro(\Set)\). Then the split extension \(S \to S \rtimes R \leftrightarrows R\) is isomorphic to a formal projective limit \(S' \to S' \rtimes R' \leftrightarrows R'\) of split extensions of unital algebras by rings in such a way that \(R \to R'\) is induces by a cofinal functor between the index categories. The functors \(\Pro(\mathbf{Rng}) \to \mathrm{Rng}(\Pro(\Set))\) and \(\Pro(\mathbf{Alg}) \to \mathrm{Alg}(\Pro(\Set))\) induce a bijection between the sets of actions of \(K \rightarrowtail R\) on \(S\) in these categories.
\end{theorem}
\begin{proof}
Without loss of generality, we may assume that the index categories of \(R\) and \(S\) are cofinite posets and \(l^*(i)\), \(r^*(i)\) have the same components \(\rho(i) \in \mathcal I_R\), \(\sigma(i) \in \mathcal I_S\) up to the order for all \(i \in \mathcal I_S\). Replacing \(\mathcal I_S\) by \(\mathcal I_R \times \mathcal I_S\) and increasing \(\rho\) and \(\sigma\), we may further assume that
\begin{itemize}
\item \(\rho\) is monotone and cofinal;
\item \(\sigma\) is monotone and \(\sigma(i) \geq i\) for all \(i\);
\item \(l_j(p, a)|_i = l_i(p|_{\rho(i)}, a|_{\sigma(i)})\) and \(r_j(a, p)|_i = r_i(a|_{\sigma(i)}, p|_{\rho(i)})\) for \(i \leq j\);
\item \(l_i\) and \(r_i\) are biadditive;
\item \(l_i(p, l_{\sigma(i)}(q, a)) = l_i(p q|_{\rho(i)}, a|_{\sigma(i)})\);
\item \(l_i(p|_{\rho(i)}, r_{\sigma(i)}(a, q)) = r_i(l_{\sigma(i)}(p, a), q|_{\rho(i)})\);
\item \(r_i(r_{\sigma(i)}(a, p), q) = r_i(a|_{\sigma(i)}, p|_{\rho(i)} q)\);
\item \(l_i(p, ab) = l_i(p, a)\, b|_i\);
\item \(r_i(a, p)\, b|_i = a|_i\, l_i(p, b)\);
\item \(r_i(ab, p) = a|_i\, r_i(b, p)\);
\item \(l_i(1, a) = a|_i = r_i(a, 1)\);
\item \(l_i(k, a) = r_i(a, k)\) for \(k \in K_{\rho(i)}\).
\end{itemize}
For simplicity we write \(\sigma^n(i)\) for \(n\) applications of \(\sigma\) to \(i\) and \(\rho \sigma^n(i)\) instead of \(\rho(\sigma^n(i))\). Let \(\widetilde S_i = R_{\rho \sigma^2(i)} \freeact S_{\sigma^3(i)}\) for all \(i \in \mathcal I_S\) and \(\widetilde S\) be the corresponding inverse system. By the above properties, there is a level morphism \(f \colon \widetilde S \to S\) of pro-rings such that
\begin{itemize}
\item \(f_i(a) = a|_i\);
\item \(f_i(p \otimes a) = l_i(p|_{\rho(i)}, a|_{\sigma(i)})\);
\item \(f_i(a \otimes p) = r_i(a|_{\sigma(i)}, p|_{\rho(i)})\);
\item \(f_i(p \otimes a \otimes q) = l_i(p|_{\rho(i)}, r_{\sigma(i)}(a|_{\sigma^2(i)}, q|_{\rho \sigma(i)}))\)
\end{itemize}
for \(a \in S_{\sigma^3(i)}\) and \(p, q \in R_{\rho \sigma^2(i)}\). Note that we take such the indices in the definition of \(\widetilde S_i\) in order for \(f\) to be multiplicative. Now let \(S'_i\) be the factor-ring of \(\widetilde S_i\) by the \(R_{\rho \sigma^2(i)}\)-invariant ideal generated by the elements
\begin{itemize}
\item \(p|_{\rho \sigma^2(i)} \otimes a|_{\sigma^3(i)} - l_{\sigma^3(i)}(p, a)\);
\item \(a|_{\sigma^3(i)} \otimes p|_{\rho \sigma^2(i)} - r_{\sigma^3(i)}(a, p)\);
\item \(k \otimes a - a \otimes k\) for \(a \in S_{\sigma^3(i)}\) and \(k \in K_{\rho \sigma^2(i)}\);
\item \(1 \otimes a - a\) for \(a \in S_{\sigma^3(i)}\).
\end{itemize}
It is easy to check that such an ideal lies in the kernel of \(f_i\). It follows that \(f\) induces a level morphism \(S' \to S\) of pro-rings, it is an isomorphism with the inverse given by the canonical maps \(S_{\sigma^3(i)} \to S'_i\). Let \(R'_i = R_{\rho \sigma^2(i)}\), it acts on \(S'_i\) by the construction as a unital algebra (with the central subring \(K'_i = K_{\rho \sigma^2(i)}\)). It is easy to check that the resulting action of \(R'\) on \(S'\) is isomorphic to the action of \(R\) on \(S\) via the isomorphisms \(R \to R'\) and \(S \to S'\).

The second claim follows from lemmas \ref{pro-ring}, \ref{pro-alg} and the first claim.
\end{proof}

As a corollary, we obtain the following result. Let \(R\) be a pro-ring, \(K\) be an abstract unital commutative ring, and suppose that the unital commutative ring object \(K\) acts on the ring object \(R\) in \(\Pro(\Set)\). Then \(R\) is isomorphic to a pro-\(K\)-algebra, i.e. the formal projective limit of \(K\)-algebras. If instead we require only that every element of \(K\) acts on the pro-ring \(R\) and these actions satisfy the classical identities, then \(R\) may be not isomorphic to a pro-\(K\)-algebra, see example \ref{k-cond}.

\begin{theorem} \label{act-pro-ring}
Let \(R\) be a unital pro-ring, \(S\) be a pro-ring, and \(l \colon R \times S \to S\), \(r \colon S \times R \to S\) be an action of the unital ring object \(R\) on the ring object \(S\) in \(\Pro(\Set)\). Then the split extension \(S \to S \rtimes R \leftrightarrows R\) is isomorphic to a formal projective limit \(S' \to S' \rtimes R' \leftrightarrows R'\) of split extensions of unital rings by rings in such a way that \(R \to R'\) is induces by a cofinal functor between the index categories. The functors \(\Pro(\mathbf{Rng}) \to \mathrm{Rng}(\Pro(\Set))\) and \(\Pro(\mathbf{Ring}) \to \mathrm{Ring}(\Pro(\Set))\) induce a bijection between the sets of actions of \(R\) on \(S\) in these categories.
\end{theorem}
\begin{proof}
Let \(K_i\) be the images of \(\mathbb Z\) in \(R_i\) for \(i \in \mathcal I_R\), then \(K \rightarrowtail R\) is a unital pro-algebra and \(l\), \(r\) is an action of the unital algebra object. Then the first claim follows from theorem \ref{act-pro-alg}. The second claim follows from the first one and lemma \ref{pro-ring}.
\end{proof}

\begin{theorem} \label{act-pro-cring}
Let \(K\) be a unital commutative pro-ring, \(S\) be a commutative pro-ring, and \(l \colon K \times S \to S\) be an action of the unital commutative ring object \(R\) on the commutative ring object \(S\) in \(\Pro(\Set)\). Then the split extension \(S \to S \rtimes R \leftrightarrows R\) is isomorphic to a formal projective limit \(S' \to S' \rtimes R' \leftrightarrows R'\) of split extensions of unital commutative rings by commutative rings in such a way that \(R \to R'\) is induces by a cofinal functor between the index categories. The functors \(\Pro(\mathbf{CRng}) \to \mathrm{CRng}(\Pro(\Set))\) and \(\Pro(\mathbf{CRing}) \to \mathrm{CRing}(\Pro(\Set))\) induce a bijection between the sets of actions of \(R\) on \(S\) in these categories.
\end{theorem}
\begin{proof}
Since \(K \rightarrowtail K\) is a unital algebra and \(l\) is an action of the unital algebra object, the first claim follows from theorem \ref{act-pro-alg}. The second claim follows from the first one and lemma \ref{pro-ring}.
\end{proof}

\begin{theorem} \label{act-pro-rng}
Let \(R\), \(S\) be pro-rings and \(l \colon R \times S \to S\), \(r \colon S \times R \to S\) be an action of the ring object \(R\) on the ring object \(S\) in \(\Pro(\Set)\). Then the split extension \(S \to S \rtimes R \leftrightarrows R\) is isomorphic to a formal projective limit \(S' \to S' \rtimes R' \leftrightarrows R'\) of split extensions of rings in such a way that \(R \to R'\) is induces by a cofinal functor between the index categories. The functor \(\Pro(\mathbf{Rng}) \to \mathrm{Rng}(\Pro(\Set))\) induces a bijection between the sets of actions of \(R\) on \(S\) in these categories.
\end{theorem}
\begin{proof}
Let \(\widetilde R_i = R_i \rtimes \mathbb Z\), then \(\widetilde R\) is a unital pro-ring and \(l\), \(r\) may be extended to the action of \(\widetilde R\) on \(S\) in the unique way. By theorem \ref{act-pro-ring}, the split extension \(S \to S \rtimes \widetilde R \leftrightarrows \widetilde R\) is isomorphic to a formal projective limit \(S' \to S' \rtimes \widetilde R' \leftrightarrows \widetilde R'\) such that \(\widetilde R' \to \widetilde R\) is induced by a cofinal functor. Then we may take \(R'_i = \Ker(\widetilde R'_i \to \mathbb Z)\) and \((S' \rtimes R')_i = \Ker((S' \rtimes \widetilde R')_i \to \mathbb Z)\). The second claim follows from the first one and lemma \ref{pro-ring}.
\end{proof}

\begin{theorem} \label{act-pro-crng}
Let \(K\), \(S\) be commutative pro-rings and \(l \colon K \times S \to S\) be an action of the commutative ring object \(K\) on the commutative ring object \(S\) in \(\Pro(\Set)\). Then the split extension \(S \to S \rtimes K \leftrightarrows K\) is isomorphic to a formal projective limit \(S' \to S' \rtimes K' \leftrightarrows K'\) of split extensions of commutative rings in such a way that \(K \to K'\) is induces by a cofinal functor between the index categories. The functor \(\Pro(\mathbf{CRng}) \to \mathrm{CRng}(\Pro(\Set))\) induces a bijection between the sets of actions of \(K\) on \(S\) in these categories.
\end{theorem}
\begin{proof}
Let \(\widetilde K_i = K_i \rtimes \mathbb Z\), then \(\widetilde K\) is a unital commutative pro-ring and \(l\) may be extended to the action of \(\widetilde K\) on \(S\) in the unique way. By theorem \ref{act-pro-cring}, the split extension \(S \to S \rtimes \widetilde K \leftrightarrows \widetilde K\) is isomorphic to a formal projective limit \(S' \to S' \rtimes \widetilde K' \leftrightarrows \widetilde K'\) such that \(\widetilde K' \to \widetilde K\) is induced by a cofinal functor. Then we may take \(K'_i = \Ker(\widetilde K'_i \to \mathbb Z)\) and \((S' \rtimes K')_i = \Ker((S' \rtimes \widetilde K')_i \to \mathbb Z)\). The second claim follows from the first one and lemma \ref{pro-ring}.
\end{proof}

\bibliographystyle{plain}  
\bibliography{references}

\end{document}